\documentclass[review]{elsarticle}

\usepackage{lineno, hyperref, amsmath, amsthm, amscd, amsfonts, amssymb, graphicx, color}
\modulolinenumbers[5]

\journal{Journal of \LaTeX\ Templates}

\begin{document}

\begin{frontmatter}

\title{QUASI-ALGEBRAS, A SPECIAL SAMPLE OF QUASILINEAR SPACES}


\author[mymainaddress
]{Reza Dehghanizade}
\ead{r.dehghanizade@stu.yazd.ac.ir}

\author[mysecondaryaddress]{Seyed Mohamad Sadegh Modarres Mosadegh\corref{mycorrespondingauthor}}
\cortext[mycorrespondingauthor]{Corresponding author}
\ead{smodarres@Yazd.ac.ir}

\address[mymainaddress]{Department of Mathematics, Yazd University, Yazd, Iran}
\address[mysecondaryaddress]{Department of Mathematics, Yazd University, Yazd, Iran}

\begin{abstract}
Similar to linear spaces, many examples of quasilinear spaces have a notion of multiplication of the elements. To characterising these examples, in the present paper we generalize the notion of quasilinear spaces and introduce quasi-algebras and normed quasi-algebras with the help of what was done in constructing algebras. After examining some properties of these new spaces, we introduce the concept of quasi-homomorphisms that is an instance of a quasilinear operator. Moreover since the study of spaces with algebraic structures cannot be completed without the study of spectrums, we define the concept of quasi-spectrum on quasi-algebras as we expect it to be.
\end{abstract}

\begin{keyword}
Quasilinear space\sep Quasi-algebra\sep Normed Quasi-algebra\sep Banach Quasi-algebra 
\MSC[2010] 46T99\sep 13J30\sep 26E25\sep 47H04
\end{keyword}

\end{frontmatter}


\section{Introduction}

\newtheorem{theorem}{Theorem}[section]
\newtheorem{lemma}[theorem]{Lemma}
\newtheorem{proposition}[theorem]{Proposition}
\newtheorem{corollary}[theorem]{Corollary}
\newtheorem{example}[theorem]{Example}
\newtheorem{question}[theorem]{Question}
\theoremstyle{definition}
\newtheorem{definition}[theorem]{Definition}
\newtheorem{problem}[theorem]{Problem}

\theoremstyle{remark}
\newtheorem{remark}[theorem]{Remark}
\numberwithin{equation}{section}

Today, linear spaces are widely used throughout mathematics. But in practice, we are always dealing with sets of vectors instead of individual vectors. So this forces us to study more about the space of subsets. In \cite{1}, Aseev presents an abstract approach to the study of space of subsets and multivalued mappings. He introduced the concept of quasilinear spaces, and used a partial order relation to extend linear spaces and generalize some fundamental results of linear algebras. After his work, some authors have motivated to extended some new results on set-valued analysis, set-valued differential equations, fuzzy quasilinear spaces, etc. For more details, see \cite{4,5,6}.\\
Aseev's work allows us to construct new spaces called quasi-algebras. In 2010, Talo and Basar introduced the concept of quasi-algebras, for the first time \cite{7}. We perfect this definition and prove some theorems and results related to these new spaces which provide us in improving some concepts in algebra and quasilinear analysis.\\
We note that many of the definitions and results in this article, can be generalized to all nonlinear and ordered spaces. In the remainder of this paper, I will pose explicit questions and problems where possible. The reader who really wants to work in this area should ask herself these questions in similar spaces.

\section{\bf QUASI-ALGEBRAS} 
 Before starting, let us give one of the most basic, simple and useful examples of quasi-algebras. It is the set $\Omega(\mathbb{R})$ of all bounded closed subsets of real numbers. The algebraic sum operation, scalar multiplication and product on it are defined pointwise and it will be assumed that the partial order on it is given by inclusion. This example can give us a general view of these spaces.
 
 \begin{question}
 What should the definition be?
 \end{question}
The definition should be such as to cover a wide range of examples as possible. moreover, by this definition, a quasi-algebra would be an example of a quasilinear space. For this purpose, as well as combining partial relation and multiplication in the best way, we propose the following definition. What you will see in it is its logical form.
\begin{definition}
A set $X$ is called a quasi-algebra if a partial order relation ``$\leq$'', an algebraic sum operation ``$+$'', an operation of multiplication by real numbers ``$\cdot$'', and a multiplication or product ``$\ast$'' are defined on it in such a way that the following conditions hold for any elements $x,y,z,v\in X$ and any $\alpha,\beta\in \mathbb{R}$:
\begin{enumerate}

  \item $x+y=y+x$;
  \item $x+(y+z)=(x+y)+z$;
  \item  There exists an element $0\in X$, called zero of X, such that $x+0=x$;
  \item $\alpha \cdot (\beta \cdot x)=(\alpha \beta)\cdot x$;
  \item $\alpha \cdot (x+y)=\alpha \cdot x+\alpha \cdot y$;
  \item $1\cdot x=x$;
  \item $0\cdot x=0$;
  \item $(\alpha +\beta)\cdot x\leq\alpha\cdot x+\beta\cdot x$;
  \item $x+z\leq y+v$ if $x\leq y$ and $z\leq v$;
  \item $\alpha\cdot x\leq\alpha\cdot y$ if $x\leq y$;
  \item $x\ast(y\ast z)=(x\ast y)\ast z$;
  \item $\alpha\cdot (x\ast y)=(\alpha\cdot x)\ast y=x\ast (\alpha\cdot y)$;
  \item $0\ast 0=0$;
  \item $x\ast (y+z)\leq x\ast y+x\ast z$ and $(x+y)\ast z\leq x\ast z+y\ast z$;
  \item $x\ast z\leq y\ast v$ if $x\leq y$ and $z\leq v$.
\end{enumerate}
\end{definition}

 In fact, a quasi-algebra is a quasilinear structure, with an operation
 $$
\ast :(x,y)\mapsto x\ast y,\hspace{0.5 cm} X\times X\longrightarrow X ,
$$
which satisfies axioms 11-15. For convenience, henceforth we write $xy$ instead of $x\ast y$.
\begin{lemma} \cite{1} \label{0isMinimal}
  In a quasi-algebra $X$, zero is minimal, i.e., $x=0$ if $x\leq 0$.
\end{lemma}
An element $x'\in X$ is called an additive-inverse of $x\in X$ if $x+x'=0$. for simplicity, from now on, we use the word inverse instead of additive-inverse. If an inverse element exists, then it is unique.
\begin{lemma}\cite{1}
\label{QuasiIsAlgebra}
  Suppose that any element $x$ in the quasi-algebra $X$ has an inverse element $x'\in X$. Then the partial order in $X$ is determined by equality, and consequently, $X$ is an algebra.
\end{lemma}
\begin{corollary}\cite{1}
  In a real algebra, equality is the only way to define a partial order such that conditions 1-15 hold.
\end{corollary}
Similar to vector spaces, it will be assumed that $-x=(-1)\cdot x$. However, here $-x$ is not necessarily the inverse of $x$. Moreover, $x-y$ means $x+(-y)$.\\
$x\in X$ is called regular if it has an inverse, otherwise it is called singular. $X_r$ and $X_s$ denote the sets of all regular and singular elements in $X$, respectively. Further, an element $x\in X$ is said to be symmetric provided that $-x=x$, and $X_d$ denotes the set of all such elements \cite{3}. It is easy to conclude that $X_r$ is a real algebra with the same operations of $X$.
\begin{lemma}
  In a quasi-algebra $X$, $x\in X_r$ if and only if $x-x=0$.
\end{lemma}
\begin{proof}
   If $x\in X_r$, there exists $x'\in X_r$ such that $x+x'=0$, since $0\leq x-x$ and $x'\leq x'$ then
$$
 x'\leq x-x+x'=-x.
$$
Also, since $0\leq x'-x'$ and $-x\leq -x$ then
$$
  -x\leq x'-x'-x=x'.
$$
Hence $x'=-x$. The other side is obvious.
\end{proof}
\begin{lemma}\cite{3}
In a quasi-algebra $X$ every regular element is minimal.
\end{lemma}
A quasi-algebra $X$ is commutative if its product is commutative.
Also, a quasi-algebra $X$ has an identity $1$ or $1_X$, if this element satisfies $1x=x1=x$ for every $x\in X$. It is easy to prove that an identity element is unique whenever it exists. A unital quasi-algebra is a quasi-algebra with an identity.
\begin{example}
  Let $A$ be a real unital algebra. Then $A$ is a unital quasi-algebra with partial order given by equality.
\end{example}

\begin{question}
How to define a norm?
\end{question}
Of course, we will generalize the norm of quasilinear spaces by introducing the notion of quasi-algebra norm in such a way that the multiplication is continuous. Also, we pursue other purposes that you will see in the examples and theorems. Similar to the definition of quasi-algebras, here we see logic in the definition, too.

\begin{definition}
  Let $X$ be a quasi-algebra. A real function $\|\cdot\|_X:X\rightarrow \mathbb{R}$ is called a norm if the following conditions hold:
\begin{enumerate}
  \item $\|x\|_X>0$ if $x\neq 0$;
  \item $\|x+y\|_X\leq\|x\|_X+\|y\|_X$;
  \item $\|\alpha\cdot x\|_X=|\alpha|\|x\|_X$;
  \item $\|xy\|_X\leq \|x\|_X\|y\|_X$;
  \item If $x\leq y$, then $\|x\|_X\leq \|y\|_X$;
  \item If for any $\epsilon>0$ there exists an element  $x_{\epsilon} \in X$ such that $x\leq y+x_{\epsilon}$ and $\|x_{\epsilon}\|_X\leq\epsilon$ then $x\leq y$.
\end{enumerate}
\end{definition}

A normed quasi-algebra is a pair $(X; \|\cdot\|_X)$, where $X$ is a non-zero quasi-algebra and $\|\cdot\|_X$ is a given quasi-algebra norm on it.
If any $x\in X$ is regular, then the concept of a normed quasi-algebra coincides with the concept of a real normed algebra, by lemma \ref{QuasiIsAlgebra}.\\
A unital normed quasi-algebra, is a normed quasi-algebra with an identity element $1$ such that $\|1\|_X=1$.

Let $X$ be a normed quasi-algebra. The Hausdorff metric on $X$ is defined by the equality
$$
  h_X(x,y)=\inf\{r\geq 0:\exists a^{r}_{1},a^{r}_{2}\in X: x\leq y+a^{r}_{1}, y\leq x+a^{r}_{2}, \|a^{r}_{i}\|_X\leq r\}.
$$
Since $x\leq y+(x-y)$ and $y\leq x+(y-x)$, so $h_X (x,y)$ is defined for any elements $x,y\in X$, and $h_X (x,y)\leq\|x-y\|_X$.
\begin{theorem}\label{continuity of operations}
  The operations of algebraic sum, product and multiplication by real numbers are continuous with respect to the Hausdorff metric. Moreover, the norm is a continuous function with respect to the Hausdorff metric.
\end{theorem}
\begin{proof}
   We prove that the operation of algebraic product is continuous. Proof of other parts is in \cite{1}.
   Suppose that $x_n\rightarrow x$ and $y_n\rightarrow y$. Then for any $\epsilon>0$ there exists an index $N$ such that the following conditions hold for $n\geq N$:
$$
  x\leq x_{n}+a_{1,n}^{\epsilon}, x_{n}\leq x+a_{2,n}^{\epsilon}, \|a_{i,n}^{\epsilon}\|_{X}\leq\min\{\frac{\epsilon}{6\|y\|},\frac{\sqrt{\epsilon}}{\sqrt{6}}\} ,
$$
$$
  y\leq y_{n}+b_{1,n}^{\epsilon}, y_{n}\leq y+b_{2,n}^{\epsilon}, \|b_{i,n}^{\epsilon}\|_{X}\leq\min\{\frac{\epsilon}{6\|x\|},\frac{\sqrt{\epsilon}}{\sqrt{6}}\} .
$$
Consequently,
$$
  xy\leq x_n y_n+x_n b_{1,n}^{\epsilon}+a_{1,n}^{\epsilon} y_n+a_{1,n}^{\epsilon} b_{1,n}^{\epsilon} 
$$
and
$$
  x_n y_n\leq xy+xb_{2,n}^{\epsilon}+a_{2,n}^{\epsilon} y+a_{2,n}^{\epsilon} b_{2,n}^{\epsilon} .
$$
Since
\begin{equation}
\begin{split}
   \| & x_n b_{1,n}^{\epsilon} +a_{1,n}^{\epsilon} y_n+a_{1,n}^{\epsilon} b_{1,n}^{\epsilon}\|_X \\
   & \leq \noindent\|x_n b_{1,n}^{\epsilon}\|_X+\|a_{1,n}^{\epsilon} y_n\|_X+\|a_{1,n}^{\epsilon} b_{1,n}^{\epsilon}\|_X \\
     & \leq \|x_n\|_X \|b_{1,n}^{\epsilon}\|_X+\|a_{1,n}^{\epsilon}\|_X \|y_n\|_X+\|a_{1,n}^{\epsilon}\|_X \|b_{1,n}^{\epsilon}\|_X \\
     & \leq (\|x\|_X+\|a_{2,n}^{\epsilon}\|_X)\|b_{1,n}^{\epsilon}\|_X+\|a_{1,n}^{\epsilon}\|_X (\|y\|_X+\|b_{2,n}^{\epsilon}\|_X)+\|a_{1,n}^{\epsilon}\|_X \|b_{1,n}^{\epsilon}\|_X\\
     & \leq\epsilon ,
\end{split}
\end{equation}
and
\begin{equation}
\begin{split}
   \|xb_{2,n}^{\epsilon} & +a_{2,n}^{\epsilon} y+a_{2,n}^{\epsilon} b_{2,n}^{\epsilon}\|_X \\
   & \leq \|xb_{2,n}^{\epsilon}\|_X+\|a_{2,n}^{\epsilon} y\|_X+\|a_{2,n}^{\epsilon} b_{2,n}^{\epsilon}\|_X \\
     & \leq \|x\|_X\|b_{2,n}^{\epsilon}\|_X+\|a_{2,n}^{\epsilon}\|_X\|y\|_X+\|a_{2,n}^{\epsilon}\| \|b_{2,n}^{\epsilon}\|_X\\& \leq\epsilon ,
\end{split}
\end{equation}
hence $x_ny_n\rightarrow xy$.
\end{proof}
\begin{theorem}
  Let $X$ be a quasi-algebra and $x,y,z\in X$, then
  $$
    h_X(xy,xz)\leq\|x\|_Xh_X(y,z).
  $$
\end{theorem}
\begin{proof}
   For any $\epsilon>0$ there exist  $a_1$ and $a_2$ such that the following conditions hold:
   $$
     y\leq z+a_1, \|a_1 \|_X\leq h_X (y,z)+\frac{\epsilon}{\|x\|_X} ,
   $$
   $$
     z\leq y+a_2, \|a_2 \|_X\leq h_X (y,z)+\frac{\epsilon}{\|x\|_X} .
   $$
Because $x\leq x$, so
\begin{equation}
\begin{split}
  xy\leq xz+xa_1, \|xa_1\|_X\leq\|x\|_X\|a_1\|_X & \leq\|x\|_X(h_X(y,z)+\frac{\epsilon}{\|x\|_X})\\
  & \leq\|x\|_X h_X(y,z)+\epsilon ,
 \end{split}
\end{equation}
\begin{equation}
\begin{split}
  xz\leq xy+xa_2, \|xa_2\|_X\leq\|x\|_X\|a_2\|_X & \leq\|x\|_X(h_X(y,z)+\frac{\epsilon}{\|x\|_X})\\
  & \leq\|x\|_X h_X(y,z)+\epsilon .
\end{split}  
\end{equation}
The theorem is proved.
\end{proof}
\begin{theorem}
  Let $X$ be a quasi-algebra and $x,y\in X$. If $y\in X_r$ then
  $$
    h_X(x,y)=\|x-y\|_X.
  $$
\end{theorem}
\begin{proof}
 For any $\epsilon>0$ there exist $a_1^{\epsilon}$ and $a_2^{\epsilon}$ such that
  $$
    x\leq y+a_1^{\epsilon}, y\leq x+a_2^{\epsilon}, \|a_i^{\epsilon}\|_X\leq h_X(x,y)+\epsilon .
  $$
  Because $y\in X_r$, so
  $$
    x-y\leq a_1^{\epsilon} .
  $$
  Thus $\|x-y\|_X\leq\|a_1^{\epsilon}\|_X\leq h_X(x,y)+\epsilon$. Hence $h_X(x,y)=\|x-y\|_X$.
\end{proof}
Only quasilinear properties of space are used to prove this theorem. So this holds for all quasilinear spaces. In fact, to equalize norm and meter, it is enough to one of the variable be regular. The following theorem is correct for all quasilinear spaces, too.
 \begin{theorem}
 Let $X$ be a quasi-algebra, $x\in X$ and $\alpha\neq 0,\pm1$ be a real number. If $x=\alpha\cdot x$ then $x=0$.
 \end{theorem}
\begin{proof}
 We may suppose that $\vert\alpha\vert > 1$, otherwise we replace $\dfrac{1}{\alpha}$ with $\alpha$. Since $x=\alpha\cdot x$ then $x=\dfrac{1}{\alpha}\cdot x$, thus $\alpha x=\dfrac{1}{\alpha}\cdot x$ and so $x=\dfrac{1}{\alpha^2}\cdot  x$. As this process continues, the following result will be obtained:
$$
x=\dfrac{1}{\alpha^2}\cdot x=\dfrac{1}{\alpha^3}\cdot x=\dfrac{1}{\alpha^4}\cdot x=\cdots .
$$
So $\dfrac{1}{\alpha^n}\cdot x \rightarrow x$ and therefore $x=0$ by theorem \ref{continuity of operations}.
\end{proof}

\begin{example}
  Let $A$ be a real Banach algebra. Then $A$ is a complete normed quasi-algebra. Here the partial order is given by equality.
Conversely, if $A$ is a complete normed quasi-algebra and any $x\in A$ is regular, then $A$ is a real Banach algebra, and the partial order on $A$ is equality. In this case $h_X (x,y)=\|x-y\|_X$.
\end{example}
A complete normed quasi-algebra is called a Banach quasi-algebra. 
\begin{example}
  Let $A$ be a real normed algebra. We show the space of nonempty closed bounded subsets of $A$ with   $\Omega(A)$. The algebraic sum operation and product on $\Omega(A)$ are defined as follows:
$$
  B+C=\overline{\{b+c:b\in B,c\in C\}} ,
$$
$$
  BC=\overline{\{bc:b\in B,c\in C\}}.
$$
Multiplication by real number $\alpha\in\mathbb{R}$ is defined by $\alpha\cdot B=\{\alpha b:b\in B\}$. Also, the partial order on $\Omega(A)$ is given by inclusion.
Then $\Omega(A)$ is a normed quasi-algera, and
$$
  \|B\|_{\Omega}=\sup_{b\in B}\|b\|_A .
$$
  We prove that $\|BC\|_X\leq\|B\|_X\|C\|_X$ for any $B$ and $C$ in $\Omega{(A)}$.
  \begin{equation}
    \begin{split}
       \|BC\|_{\Omega}=\sup_{t\in \overline{\{bc:b\in B,c\in C\}}}\|t\|_A & =\sup_{bc\in \{bc:b\in B,c\in C\}}\|bc\|_A \\
         & \leq\sup_{bc\in \{bc:b\in B,c\in C\}}\|b\|_A\|c\|_A \\
         & \leq\sup_{b\in B}\|b\|_A\sup_{c\in C}\|c\|_A \\
         & =\|B\|_X\|C\|_X .
    \end{split}
  \end{equation}
  Proof of other parts are easy.
Suppose that $S_r(0)$ is the closed ball of radius $r$ about $0\in A$. Then 
$$
  h_{\Omega}(B,C)=\inf{\{r\geq0:B\subseteq C+S_r(0), C\subseteq B+S_r(0)\}} ,
$$
defines the Hausdorff metric on $\Omega{(A)}$.
\end{example}
\begin{example}
  Consider that $S$ is a compact topological space, and $X$ is a normed quasi-algebra with an identity element $1$. We denote by $C(S,X)$ the space of all continuous mappings $f:S\rightarrow X$. We write $f_1\leq f_2$ if and only if $f_1(s)\leq f_2(s)$ for any $s\in S$. Also, the operations of algebraic sum, product and multiplication by real numbers are defined pointwise.
   $\|f\|_C=\max_{s\in S}{\|f(s)\|_X}$ defines a norm on $C(S,X)$. 
 Then it is a normed quasi-algebra. The constant function $1_X$ is the identity of $C(S,X)$, in fact $1_C(s)=1_X$ for any $s\in S$.
\end{example}

\begin{problem}
Now’s the time to define `ideal's. Proper definition of an ideal will open many closed doors for us. But the question is how. The following definitions may be helpful.
\begin{definition}
  Let $X$ be a quasilinear space, $s\in X$ and $S\subseteq X$. $s$ is called a l-member of $S$ and it is denoted by $s{\in}_{\leq } S$ if there exists $r\in S$ such that $s\leq r$.
Likewise, $t$ is called a g-member of $S$ and it is denoted by $t{\in}_{\geq} S$ if there exists $r\in S$ such that $r\leq t$. ``l'' and ``g'' are taken from the words ``less'' and ``greater'', respectively.
\end{definition}
\begin{definition}
  Let $X$ be a quasilinear space and $S,T\subseteq X$. $T$ is called a l-subset of $S$ and it is denoted by $T{\subset}_{\leq } S$ if $t{\in}_{\leq } S$ for any $t\in T$. a g-subset is defined similarly.
\end{definition}
Given that $I$ is an ideal of an algebra $A$ if (i) I is a vector subspace of $A$ and (ii) both $AI\subseteq I$  and $IA\subseteq I$, can we achieve a definition in quasi-algebras by replacing $\subseteq$ with ${\subset}_{\leq }$? Note that a vector subspace is defined in quasilinear spaces similar to linear spaces \cite{3}.

\end{problem}

\section{\bf BUILD A CHAIN OF SINGULAR ELEMENTS}
In an algebra, an element minus itself is equal to zero. But that doesn't always happen in any quasi-algebra. By definition of these spaces, zero becomes smaller than subtracting an element from itself. However, elements minus themselves can play a greater role in identifying and analyzing quasi-algebras. In this section, we try to make an infinite chain of elements.\\
Let’s start with a simple example. consider the space of all nonempty convex compact subsets of $\mathbb{R}$ denoted by $\Omega_c (\mathbb{R})$. For every $a,b\in\mathbb{R}$, the following statement is true:
$$
[a,b]\subseteq [a,b]+([a,b]-[a,b])
$$
This method help us to build a chain of unequal elements, and we will do it here. Note that all results of this section are valid for the entire quasilinear spaces and used to prove an important theorem in the spectrum section.
\begin{lemma}\cite{1} \label{x_n<y_nThenx_0<y_0}
Let  $X$ be a quasi-algebra. Suppose that $x_n\rightarrow x_0$ and $y_n\rightarrow y_0$, and that $x_n\leq y_n$ for any positive integer $n$. Then $x_0\leq y_0$.
\end{lemma}
If $x\leq y$ and $x\neq y$ we represent it with $x\lneqq y$.
\begin{lemma}
 Let $X$ be a quasi-algebra and $x\in X_s$. Then $x\lneqq x+x-x$.
\end{lemma}
\begin{proof}
Assume to the contrary that the result is false. Then $x=x+x-x$. Thus $x-x\geq 2\cdot(x-x)$ and so $1/2\cdot(x-x)\geq x-x$. With a little effort, we get the following condition for any $n\in\mathbb{N}$
$$
1/2^n\cdot(x-x)\geq x-x .
$$
Therefore $x-x=0$ by lemma \ref{x_n<y_nThenx_0<y_0}, lemma \ref{0isMinimal} and theorem \ref{continuity of operations}, contrary to the fact that $x\in X_s$.
\end{proof}
\begin{lemma}\label{x+yinX_r->x,yinX_r}\cite{10}
 Let $X$ be a quasi-algebra. For every $x,y\in X$, $x+y\in X_r$ implies $x,y\in X_r$.
\end{lemma}
\begin{theorem}\label{ThereExistsSingularElement}
 Let $X$ be a quasi-algebra and $x\in X$. If there exists no $y\in X_s$ such that $x\leq y$, then $X$ is an algebra.
\end{theorem}
You can see the proof of this  theorem in \cite{8}. Here, we prove it easier, in a different way.
\begin{proof}
 Since $x\leq x+y-y$ for any $y\in X$, then $x+y-y\in X_r$ by assumption. So $y\in X_r$ by lemma \ref{x+yinX_r->x,yinX_r}. Thus $X$ is an algebra by lemma \ref{QuasiIsAlgebra}.
\end{proof}
\begin{theorem}\label{chainOfsingular}
Suppose that $X$ is a non-linear quasi-algebra. For every $x\in X$, there exist $x_1,x_2,x_3,\cdots\in X$ such that
$$
 x\lneqq x_1\lneqq x_2\lneqq x_3\lneqq\cdots .
$$
\end{theorem}
\begin{proof}
 If $x\in X_s$, then we set $x_1=x+x-x$, $x_2=x_1+x_1-x_1$, $x_3=x_2+x_2-x_2$ and the rest is defined likewise.
 If $x\in X_r$, then there exists at least one $y\in X_s$ such that $x\leq y$ by theorem \ref{ThereExistsSingularElement}. So we set $x_1=y$ and define the rest as in the first case.
\end{proof}

\section{\bf MAPPINGS IN QUASI-ALGEBRAS}

As with studying any mathematical object, we ought to set quasilinear spaces and quasi-algebras in a categorical setting by specifying the morphisms between them.

\subsection{\bf QUASI-HOMOMORPHISMS}
Before we get started, let's take a look at the definition of the quasilinear operator which may be founded in \cite{1}. The following is the first and most logical definition that comes to mind for an operator in quasi algebras.

\begin{definition}
  Let X and Y be quasi-algebras. A mapping $\varphi:X\rightarrow Y$ is a quasi-homomorphism if it satisfies the following conditions:
  \begin{enumerate}
    \item $\varphi(\alpha\cdot x)=\alpha\cdot\varphi(x)$;
    \item $\varphi(x+y)\leq\varphi(x)+\varphi(y)$;
    \item $\varphi(xy)\leq\varphi(x)\varphi(y)$;
    \item if $x\leq y$, then $\varphi(x)\leq\varphi(y)$.  
  \end{enumerate}
  In fact, $\varphi$ is a quasilinear operator which satisfies condition 3.
\end{definition}
Condition 4 is the order preserving property and a map that has this condition is called an order preserving mapping.\\
If X and Y are algebras, then the definition of a quasi-homomorphism coincides with the usual definition of an algebra-homomorphism.
\begin{example}
  Let X be an algebra, and let the function $p:X\rightarrow\mathbb{R}$ be a sublinear such that
  $$
    p(xy)\leq p(x)p(y) .
  $$
   Let $\Omega_c(\mathbb{R})\subseteq\Omega(\mathbb{R})$ be the set of all bounded closed convex subsets of $\mathbb{R}$ ordered by inclusion, see \cite{1}. Then the mapping $\varphi:X\rightarrow\Omega_c(\mathbb{R})$ defined by $\varphi(x)=[-p(-x),p(x)]$ is a quasi-homomorphism from $X$ to $\Omega_c(\mathbb{R})$. Note that the operations on $\Omega_c(\mathbb{R})$ are defined as the operations on $\Omega(\mathbb{R})$.
  We prove that $\varphi(xy)\subseteq\varphi(x)\varphi(y)$.
  \begin{equation}
  \begin{split}
     \varphi(xy) & =[-p(-xy),p(xy)] \\
       & \subseteq[-p(x)p(-y),p(x)p(y)] \\
       & =p(x)[-p(-y),p(y)] \\
       & \subseteq[-p(-x),p(x)][-p(-y),p(y)]=\varphi(x)\varphi(y) .
  \end{split}
  \end{equation}
  Other parts of proof are easy.
\end{example}
\begin{definition}
  Let $X$ and $Y$ be normed quasi-algebras. A quasi-homomorphism $\varphi:X\rightarrow Y$ is said to be bounded if there exists a number $k>0$ such that $\|\varphi(x)\|_Y\leq k\cdot\|x\|_X$ for any $x\in X$.
\end{definition}
\begin{theorem}
  Suppose that $X$ and $Y$ are normed quasi-algebras. Then a quasi-homomorphism $\varphi:X\rightarrow Y$ is bounded if and only if it is continuous at the point $0\in X$. The continuity of $\varphi$ at $0$ implies that it is uniformly continuous on $X$.
\end{theorem}
\begin{proof}
  see \cite{1}.
\end{proof}
Let $X$ and $Y$ be normed quasi-algebras. We denote by $\Phi(X,Y)$ the space of all bounded quasi-homomorphisms from $X$ to $Y$. It will be assumed that $\varphi_1\leq\varphi_2$ if and only if $\varphi_1(x)\leq\varphi_2(x)$ for any $x\in X$. Moreover, the operation of multipllication by real numbers, algebraic sum and product are defined by:
$$
(\alpha\cdot\varphi)(x)=\alpha\cdot\varphi(x) ,
$$
$$
  (\varphi_1+\varphi_2)(x)=\varphi_1(x)+\varphi_2(x) ,
$$
$$
  (\varphi_1 \varphi_2 )(x)=\varphi_1(\varphi_2(x)) .
$$
Since, if $\varphi_1\leq\varphi_2$ and $\varphi_3\leq\varphi_4$ then $\varphi_1(\varphi_3(x))\leq\varphi_2(\varphi_3(x))\leq\varphi_2(\varphi_4(x))$ for any $x\in X$, thus
$$
  \varphi_1\varphi_3\leq\varphi_2\varphi_4 .
$$
Then $\Phi(X,Y)$ is a quasi-algebra. Checking other properties is easy.
The norm on $\Phi(X,Y)$ is defined by
$$
  \|\varphi\|_{\Phi}=\sup_{\|x\|_X=1}\|\varphi(x)\|_Y .
$$
Then $\Phi(X,Y)$ is a normed quasi-algebra.

Now may be a  suitable occasion to ask this question:
\begin{question}
Will we be able to generalize the notion of character to quasi-algebras?
\end{question}

\begin{theorem}
  Let $A$ be an unital algebra. Then there is no quasi-homomorphism from $\Omega(A)$ to $\mathbb{R}$.
\end{theorem}
\begin{proof}
  Let $\varphi:\Omega(A)\rightarrow\mathbb{R}$ be a quasi-homomorphism. Since the partial order in $\mathbb{R}$ is determined by equality, $\varphi(\{1\})=\varphi(\{1\})\varphi(\{1\})$, so $\varphi(\{1\})=1$. Moreover $\varphi(\{0\})=\varphi(0\cdot\{1\})=0\cdot\varphi(\{1\})=0$. If $B=\{1,0\}$, applying the quasi-homomorphism $\varphi$ then gives $\varphi(B)=\varphi(\{1\})=1$ and $\varphi(B)=\varphi(\{0\})=0$, a contradiction. Thus a quasi-homomorphism for sets containing $0$ and $1$ can not be defined.
\end{proof}
This is our motive for the following definition.
\begin{definition}
  Let $X$ be a quasi-algebra. Then a quasi-character on X is a non-zero quasi-homomorphism $\varphi:X\rightarrow\Omega(\mathbb{R})$.
\end{definition}
\begin{example}
  Let $m,k\in\mathbb{N}$, and $2\leq k$, then the mapping $\varphi:\mathbb{R}\rightarrow\Omega(\mathbb{R})$ defined by
  $$
    \varphi(x)=\{0\}\cup\{\frac{x}{m^n}:n\in\mathbb{N}\}\cup\{xm^n:1\leq n\leq k\}
  $$
  is a continuous quasi-character. We check the quasi-homomorphism conditions. It is clear that $\varphi(ax)=a\varphi(x)$ and $\varphi(x+y)\subseteq\varphi(x)+\varphi(y)$ For any $a,x,y\in\mathbb{R}$. Since
  $$
    \varphi(xy)=\{0\}\cup\{\frac{xy}{m^n}:n\in\mathbb{N}\}\cup\{xym^n:1\leq n\leq k\} ,
  $$
  and $0=0\times0$, $\frac{xy}{m^{n}}=\frac{x}{m^{n+1}}\times ym$, $xym=\frac{x}{m}\times ym^{2}$ and ${xym^{n}}_{n>1}=xm\times ym^{n-1}$, hence $\varphi(xy)\subseteq\varphi(x)\varphi(y)$. Proof of continuity is simple.
\end{example}
This example shows that it is not necessarily required $\varphi(1)=1$ or $1\leq\varphi(1)$.
\begin{example}\label{ex 37}
  Let $A$ be a real algebra and $n\in\mathbb{N}$. For $i=1,\ldots,n$, let $\varphi_i:A\rightarrow\mathbb{R}$ be a non-zero algebra-homomorphism, then
  $$
    \varphi(x)=\{\varphi_1(x),\varphi_2(x),\ldots,\varphi_n(x)\}
  $$
  is a quasi-character on $A$.
\end{example}
\begin{lemma}\label{(x_i)->{x_i}continuous}
  The mapping
 $$
    (x_1,x_2,\ldots ,x_n)\mapsto\{x_1,x_2,\ldots ,x_n\},\hspace{0.5 cm}{\mathbb{R}}^n\rightarrow\Omega(\mathbb{R})
  $$
  is continuous.
\end{lemma}
\begin{proof}
  It is obvious.
\end{proof}
\begin{theorem}
  Let $A$ be a real Banach algebra. Let $\varphi$ be a quasi-character on $A$, as defined in example \ref{ex 37}. Then $\varphi$ is continuous and $\|\varphi\|\leq 1$. If $A$ is unital, then $\|\varphi\|=1$.
\end{theorem}
\begin{proof}
  Suppose that every $\varphi_i:A\rightarrow\mathbb{R}$ is a non-zero algebra-homomorphism on A for $i=1,\ldots,n$ and let  $\varphi(x)=\{\varphi_1(x),\varphi_2(x),\ldots ,\varphi_n(x)\}$. Then, by lemma \ref{(x_i)->{x_i}continuous} ,$\varphi$ is continuous because every $\varphi_i$ is continuous \cite{2}.
  Also, since $\|\varphi_i\|\leq 1$, thus
  $$
    \|\varphi\|=\sup_{\|x\|_A=1} \|\varphi(x)\|_{\Omega}=\sup_{\|x\|_A=1}{\max\{|\varphi_1(x)|,|\varphi_2(x)|,\ldots,|\varphi_n(x)|\}}\leq 1 .
  $$
  If $A$ is unital then  $\varphi_i(1)=1$, therefor
  $$
  |\varphi(1)\|_{\Omega}=\|\{\varphi_1(1),\varphi_2(1),\ldots,\varphi_n(1)\}\|_{\Omega}=\|\{1\}\|_{\Omega}=1 .
  $$
  So $\|\varphi\|=1$.
\end{proof}

\subsection{\bf ORDER PRESERVING IN REVERSE}
Quasilinear operators and quasi-homomorphisms preserve the order. But there are examples, so that the inverse of a map has this property.
\begin{definition}
 Let $X$ and $Y$ be quasilinear spaces and $\varphi:X\rightarrow Y$ be a mapping. $\varphi$ is order preserving in reverse if $\varphi(x)\leq\varphi(y)$ concludes $x\leq y$ for every $x,y\in X$. In this case, we say $\varphi$ is an opr mapping.
\end{definition}
\begin{example}\label{iopExample}
Let $X={\Omega}_C(\mathbb{R})$ and
$$
\varphi :([a,b])\mapsto([\dfrac{(a+b)}{2}-\dfrac{(b-a)}{4},\dfrac{(a+b)}{2}+\dfrac{(b-a)}{4}]),\hspace{0.5cm} X\rightarrow X .
$$
Then $\varphi$ is an opr mapping. Geometrically, this function converts an interval into an interval with the same center and half length. Suppose that $A=[-2,2]$ and $B=[-4,4]$. Since $\varphi(A)=[-1,1]$, $\varphi(B)=[-2,2]$ and $\varphi(AB)=\varphi([-8,8])=[-4,4]$, so $\varphi(AB)\nsubseteq\varphi(A)\varphi(B)$. Therefore $\varphi$ is not a quasi-homomorphism. Of course, we could prove that this mapping does not have the order preserving property.\\
Let
$$
\psi :([a,b])\mapsto([\dfrac{(a+b)}{2}-(b-a),\dfrac{(a+b)}{2}+(b-a)]),\hspace{0.5cm}X\rightarrow X .
$$
 Since $\lbrace 3\rbrace\subseteq [-4,4]$ but  $\lbrace 3\rbrace\nsubseteq [-2,2]$ then $\psi$ is not an opr quasi-homomorphism. $\psi$ converts an interval into an interval with the same center and doubled in length.\\
Let
$$
\rho :([-a,a])\mapsto([-2a,2a]),\hspace{0.5cm}X_d\rightarrow X_d .
$$
Then $\rho$ is an opr quasi-homomorphism.
\end{example}
\begin{lemma}
Let $X$ and $Y$ be quasilinear spaces and $\varphi:X\rightarrow Y$ be an opr mapping. Then $\varphi$ is injective.
\end{lemma}
\begin{proof}
It is obvious by the reflexive property of the partial order relation.
\end{proof}
\begin{theorem}
Suppose that $X$ and $Y$ are quasi-algebras and $\varphi :X\rightarrow Y$ is a surjective opr quasi-homomophism. Then for every $y_1,y_2\in Y$ and $\alpha\in\mathbb{R}$, the following conditions hold:
\begin{equation}
\begin{split}
& {\varphi}^{-1}(\alpha\cdot y_1)=\alpha\cdot{\varphi}^{-1}(y_1) ,\\
& {\varphi}^{-1}(y_1)+{\varphi}^{-1}(y_2)\leq{\varphi}^{-1}(y_1+y_2) ,\\
& {\varphi}^{-1}(y_1){\varphi}^{-1}(y_2)\leq{\varphi}^{-1}(y_1 y_2) .
\end{split}
\end{equation}
\end{theorem}
\begin{proof}
 Let $\varphi(x_1)=y_1$ and $\varphi(x_2)=y_2$, then
 $$
{\varphi}^{-1}(\alpha\cdot y_1)={\varphi}^{-1}(\alpha\cdot\varphi(x_1))={\varphi}^{-1}(\varphi(\alpha\cdot  x_1))=\alpha\cdot x_1=\alpha\cdot{\varphi}^{-1}(y_1) .
 $$
Also, since $\varphi(x_1+x_2)\leq\varphi(x_1)+\varphi(x_2)$ and $\varphi$ is opr, so
\begin{equation}
\begin{split}
{\varphi}^{-1}(y_1)+{\varphi}^{-1}(y_2)=x_1+x_2 & ={\varphi}^{-1}(\varphi(x_1+x_2))\\
& \leq{\varphi}^{-1} (\varphi(x_1)+\varphi(x_2))\\
& \leq{\varphi}^{-1}(y_1+y_2) .
\end{split}
\end{equation}
The last part proves in a similar way.
\end{proof}
\begin{lemma}\label{phixInYrThenxInXr}
Let $X$ and $Y$ be quasi-algebras and $\varphi :X\rightarrow Y$ is a mapping such that 
$$
\varphi(\alpha\cdot x)=\alpha\cdot\varphi(x) ,
$$
$$
\varphi(x+y)\leq\varphi(x)+\varphi(y) ,
$$
and $ker (\varphi)=\lbrace 0\rbrace$. If $\varphi(x)\in Y_r$, then $x\in X_r$.
\end{lemma}
\begin{proof}
 Since $\varphi(x-x)\leq \varphi(x)-\varphi(x)=0$, then $\varphi(x-x)=0$ by lemma \ref{0isMinimal}. So x-x=0.
\end{proof}
It is obvious that if $\varphi$ is an injective quasi-homomorphism, then $ker (\varphi)=\lbrace 0\rbrace$. But the converse is not true necessarily.

\section{\bf THE GROUP OF UNITS}
Let $X$ be a unital quasi-algebra. Then $x\in X$ is left product-invertible (right product-invertible) if there exists $y\in X$ with $yx=1$ ($xy=1$). Furthermore, $x\in X$ is product-invertible if it is both left and right product-invertible. In this case, there exists a unique element $y\in X$ such that $yx=xy=1$. $y$ is called the product-inverse of $x$\ and it is denoted by $x^{-1}$.\\ We also say that $x$ is a unit of $X$ when it is product-invertible, and write $G=G(X)$ for the set of all these units. Since $1\in G(X)$ and $(xy)^{-1}=y^{-1} x^{-1}$ for $x,y\in G(X)$, then $(G(X),\ast)$ is a group.\\
\begin{example}
   Let $X=\Omega(\mathbb{R})$.The set of all singletons of real non-zero numbers $\{\{a\}:a\in\mathbb{R}\setminus\{0\}\}$ is the set of all product-invertible members of $X$. For every $r>0$ and $x\in G(X)$, we have
   $$
     [x-\frac{r}{2},x+\frac{r}{2}]\in\{y\in X:h_X(x,y)<r\}.
   $$
   So $G(X)$ is not open in $X$.
\end{example}
\begin{lemma}\label{equal}
  Let $X$ be a quasi-algebra with an identity $1$. If $x\in G(X)$ then $x(y+z)=xy+xz$ for every $y,z\in X$.
\end{lemma}
\begin{proof}
  Since $xy+xz\leq xy+xz$, thus
  $$
    xy+xz=1(xy+xz)=(xx^{-1})(xy+xz)=x(x^{-1}(xy+xz))\leq x(y+z) .
  $$
\end{proof}
This seemingly simple lemma will be the key to solving many problems.

\begin{corollary}\label{xinGthenxinXr}
  Let $X$ be a quasi-algebra with an identity $1$. If $1\in X_r$ then $x\in X_r$ for any $x\in G(X)$.
\end{corollary}
\begin{proof}
  It follows from lemma \ref{equal} that $x(1-1)=x-x$. Since $x(1-1)=x 0=0$, thus $x-x=0$.
\end{proof}
\begin{theorem}
   Let $X$ be a commutative unital normed quasi-algebra. Then the mapping
   $$
     x\rightarrow x^{-1},\hspace{0.5 cm} G(X)\rightarrow G(X)
   $$
is a homeomorphism of $G(X)$ onto itself.
\end{theorem}
\begin{proof}
  Set $\omega(x)=x^{-1} (x\in G(X))$. Then certainly $\omega$ is a bijective mapping of  $G(X)$ onto itself. Also, since ${\omega}^{-1}=\omega$, it suffices to prove that $\omega$ is continuous.

  Let $x\in G(X)$. For every $\epsilon>0$, if
  $$
    \delta<\min{\{\frac{1}{2\|x^{-1}\|},\frac{\epsilon}{(2(\|1-1\|+1){\|x^{-1}\|}^2 )}\}}
  $$
  for $y\in G(X)$ with $h_X(x,y)<\delta$, there are $c_1,c_2\in X$ such that
  $$
    x\leq y+c_1, \|c_1\|<\delta ,
  $$
  $$
    y\leq x+c_2, \|c_2\|<\delta .
  $$
  Thus
  $$
    x-y\leq y-y+c_1 .
  $$
  Since $x^{-1},y^{-1}\in G(X)$, it follows from lemma \ref{equal} that 
  $$
  y^{-1}-x^{-1}=y^{-1}(x-y)x^{-1}.
  $$
  So
  \begin{equation}
  \begin{split}
      \|y^{-1}\| & \leq\|y^{-1}-x^{-1}\|+\|x^{-1}\| \\
       & =\|y^{-1}(x-y)x^{-1}\|+\|x^{-1}\| \\
       & \leq\|y^{-1}(y-y+c_1)x^{-1}\|+\|x^{-1}\| \\
       & \leq\|(1-1+y^{-1}c_1)x^{-1}\|+\|x^{-1}\| \\
       &  \leq(\|1-1\|+\|y^{-1}\|\|c_1\|)\|x^{-1}\|+\|x^{-1}\| \\
       & \leq(\|y^{-1}\|\|c_1\|)\|x^{-1}\|+(\|1-1\|+1)\|x^{-1}\| \\
       & \leq\frac{\|y^{-1}\|}{2}+(\|1-1\|+1)\|x^{-1}\| .
  \end{split}
  \end{equation}
  Hence $\|y^{-1}\|\leq 2(\|1-1\|+1)\|x^{-1}\|$.

  Also, because $x^{-1}y^{-1}\leq x^{-1}y^{-1}$ then
  $$
    x^{-1}y^{-1}x\leq x^{-1}y^{-1}y+x^{-1}y^{-1}c_1, \|x^{-1}y^{-1}c_1\|\leq\|x^{-1}\|\|y^{-1}\|\|c_1\|<\epsilon ,
  $$
  $$
    x^{-1}y^{-1}y\leq x^{-1}y^{-1}x+x^{-1}y^{-1}c_2, \|x^{-1}y^{-1}c_2\|\leq\|x^{-1}\|\|y^{-1}\|\|c_2\|<\epsilon ,
  $$
  so $h_X(x^{-1},y^{-1})<\epsilon$, and the result follows.
\end{proof}

\section{\bf QUASI-SPECTRUM}
Spectrum is an integral part of an algebra. Spectral theory, which has numerous applications, should also be extended to quasi-algebras. So we introduce the concept of quasi-spectrum that is a generalisation of algebra spectrum.
\begin{definition}
Let $A$ be a real algebra with an identity, and let $a\in A$. Then the spectrum of $a$, denoted by ${Sp}_A (a)$ (or usually $Sp(a)$) is defined as:
$$
{Sp}_A (a)=\lbrace\lambda\in\mathbb{R}:\lambda 1-a\notin G(A)\rbrace .
$$
\end{definition}
\begin{definition}
Let $X$ be a quasi-algebra with an identity $1\in X_r$, and let $x\in X$. Then the quasi-spectrum of $x$, denoted by ${QSp}_X (x)$ or $QSp(x)$ is defined as:
$$
{QSp}_X (x)={\cup}_{(t\in{X_r},t\leq x)} {Sp}_{X_r} (t)=\lbrace {Sp}_{X_r}(t):t\in X_r,t\leq x\rbrace .
$$
\end{definition}

\begin{question}
Where did the idea for this definition come from?
\end{question}
We wanted the quasi-spectrum of a subset of an algebra to be equal to the union of the spectrum of its elements. Also, by involving the spectrum in the definition, we make the most of its benefits. In the following examples and theorems, you will see its effects.

\begin{example}
Let $A$ be a real algebra with an identity, and let $a\in A$. We already saw that $A$ is a quasi-algebra with the partial order defined by equality and then
$$
{QSp}_A(a)={Sp}_A(a) .
$$
\end{example}
\begin{example}
Let $X=\Omega(\mathbb{R})$ and $x\in X$. then ${QSp}_X(x)=x$.
\end{example}
\begin{lemma} \cite{2} \label{sp(ab)in0sp(ba)}
Let $A$ be a real unital algebra, and $a,b\in A$. Then
$$
Sp (ab)\subseteq\lbrace 0\rbrace \cup Sp (ba).
$$
\end{lemma}
We note that, in this section, from now on, all quasi-algebras have an identity, and it is regular.
\begin{theorem}
Let $X$ be a quasi-algebra, and $x,y\in X$. If $x\in G(X)$ Then
$$
QSp(xy)\subseteq\lbrace 0\rbrace\cup QSp(yx) .
$$
\end{theorem}
\begin{proof}
Suppose that $\lambda\neq 0$ is an element of $QSp(xy)$. So there exists $z\in X_r$ such that $z\leq xy$ and $\lambda\in Sp(z)$. Since
$$
0=x^{-1} 0 x=x^{-1} (z-z)x=x^{-1} zx-x^{-1} zx
$$
by lemma \ref{equal}. Then $x^{-1} zx\in X_r$ and so $\lambda\in Sp(x^{-1}zx)$ by lemma \ref{sp(ab)in0sp(ba)}. Because $x^{-1} zx\leq yx$ then $\lambda\in QSp(yx)$.
\end{proof}
\begin{lemma}\label{phi(1-x)=1-phi(x)}
Let $\varphi :X\rightarrow Y$ be a mapping such that the following conditions hold
\begin{enumerate}
  \item $\varphi(\alpha\cdot x)=\alpha\cdot\varphi(x)$,
  \item $\varphi(x+y)\leq\varphi(x)+\varphi(y)$,
  \item  $\varphi(1)\leq\varphi(1)\varphi(1)$,\\
 moreover
  \item $\varphi(G(X))\subseteq G(Y)$.
\end{enumerate}
 If $\varphi(x)\in Y_r$ then $\varphi(1-x)=1-\varphi(x)$.
\end{lemma}
\begin{proof}
$\varphi(1)\in G(Y)$ then $\varphi(1)\in Y_r$ by corollary \ref{xinGthenxinXr}. Since $\varphi(1)\leq \varphi(1)\varphi(1)$, then $1\leq\varphi(1)$. So $\varphi(1)=1$ by minimality of regular elements. Since $\varphi(1)-\varphi(x)$ is minimal hence
$$
1-\varphi(x)=\varphi(1)-\varphi(x)=\varphi(1-x) .
$$
\end{proof}
\begin{theorem}\label{QspphixInQspx}
Let $X$ and $Y$ be quasi-algebras and $\varphi :X\rightarrow Y$ be an opr mapping such that four conditions of lemma \ref{phi(1-x)=1-phi(x)} are true. Then
$$
{QSp}_Y (\varphi(x))\subseteq {QSp}_X (x)\cup({\cup}_{(y\in Y')}{Sp}_Y (y))\hspace{1cm}(x\in X) ,
$$
when $Y'={(Im(\varphi))}^c\cap Y_r$.
\end{theorem}
\begin{proof}
If $\lambda\in{QSp}_Y \varphi (x)$ and $\lambda\notin{\cup}_{(y\in Y')}{Sp}_Y y$ , so there exists $\varphi(t)\in Y_r$  such that $\varphi(t)\leq \varphi(x)$ and $\lambda\in{Sp}_Y \varphi(t)$. Thus $\lambda 1-\varphi(t)\notin G(Y)$. Therefor $\lambda 1-t\notin G(X)$ by lemma \ref{phi(1-x)=1-phi(x)}, and because $t\leq x$ by opr property, and since $t\in X_r$ by lemma \ref{phixInYrThenxInXr}, then $\lambda\in {Sp}_X t$.
\end{proof}
\begin{example}
Consider the function $\varphi$ which was defined in example \ref{iopExample}. Since $Im (\varphi)=X$, then
$$
{QSp}_X \varphi(x)\subseteq {QSp}_X x.
$$
\end{example}
\begin{corollary}
The result of theorem \ref{QspphixInQspx} is true when $\varphi :X\rightarrow Y$ is an opr quasi-homomorphism such that $\varphi(G(X))\subseteq G(Y)$.
\end{corollary}
\begin{example}
Suppose that $X={\Omega}_C(\mathbb{R})$, $Y={\Omega}_C(\mathbb{C})$ and
$$
\varphi([a,b])\mapsto\lbrace z\in \mathbb{C}:\Vert z-\dfrac{a+b}{2}\Vert\leq\dfrac{b-a}{2}\rbrace,\hspace{0.5cm}   X\rightarrow Y .
$$
Let $Y'={(Im (\varphi))}^c\cap Y_r$. Since $({\cup}_{(y\in Y')}{Sp}_Y y)=\emptyset$ so
$$
{QSp}_X \varphi(x)\subseteq {QSp}_X x .
$$
We prove that $\varphi$ is a quasi-homomorphism. Suppose that $A_1,A_2\in {\Omega}_C (\mathbb{R})$, also $\varphi(A_1)$ and $\varphi(A_2)$ are circles with center $a_1\in\mathbb{R}$ and $a_2\in\mathbb{R}$, respectively. Because $A_1\subseteq\varphi(A_1)$  and $A_2\subseteq\varphi(A_2)$, so
$$
A_1 A_2\subseteq \varphi(A_1)\varphi(A_2).
$$
Moreover
\begin{equation}
\begin{split}
\varphi(A_1)\varphi(A_2) & =(\varphi(A_1)-a_1+a_1)(\varphi(A_2)-a_2+a_2) \\
& =(\varphi(A_1)-a_1)(\varphi(A_2)-a_2)+(\varphi(A_1)-a_1)a_2\\
& \hspace*{0.4 cm} +(\varphi(A_2)-a_2)a_1+a_1 a_2 .
\end{split}
\end{equation}
Thus $\varphi(A_1)\varphi(A_2)$ is a circle with center on the axis of real numbers . Then it contains the circle of diameter $A_1 A_2$ and so
$$
\varphi(A_1 A_2)\subseteq\varphi(A_1)\varphi(A_2).
$$
 Proof of other parts is similar and left to the reader.
\end{example}
\begin{lemma} \label{XrIsB-algebra}
Let $X$ be a Banach quasi-algebra. Then $X_r$ is a real Banach algebra.
\end{lemma}
\begin{proof}
 Suppose that ${x_i}$ is a Cauchi sequence in $X_r$. Since $X$ is complete, there exists $x\in X$ such that $x_i\rightarrow x$ and $-x_i\rightarrow -x$. Thus $x_i-x_i\rightarrow x-x$ and since $x_i-x_i\rightarrow 0$, so $x-x=0$. Therefore $x\in X_r$.	
\end{proof}
\begin{lemma}\cite{2} \label{spa<|a|}
Let $A$ be a real Banach algebra, and let  $a\in A$. Then $Sp(a)$ is a subset of the disk $\lbrace\lambda\in\mathbb{R}:\vert\lambda\vert\leq\Vert a\Vert\rbrace$.
\end{lemma}
\begin{theorem}
Let $X$ be a Banach quasi-algebra and $x\in X$. Then
$$
{QSp}_X (x)\subseteq\lbrace\lambda\in\mathbb{R}:\vert\lambda\vert\leq{\Vert x\Vert}_X\rbrace .
$$
\end{theorem}
\begin{proof}
 $X_r$ is a Banach algebra by lemma \ref{XrIsB-algebra}, so 
 $$
 {QSp}_X (x)={\cup}_{(t\in X_r,t\leq x)}{Sp}_{X_r}(t)\subseteq{\cup}_{(t\in X_r,t\leq x)}\lbrace\lambda\in\mathbb{R} :\vert\lambda\vert\leq{\Vert t\Vert}_X\rbrace 
 $$
by lemma \ref{spa<|a|}. Since ${\Vert t\Vert}_X\leq {\Vert x\Vert}_X$, then
$$
{\cup}_{(t\in X_r,t\leq x)} \lbrace\lambda\in\mathbb{R} :\vert\lambda\vert\leq{\Vert t\Vert}_X\rbrace\subseteq\lbrace\lambda\in\mathbb{R} :\vert\lambda\vert\leq{\Vert x\Vert}_X\rbrace .
$$
\end{proof}
 \begin{theorem}
 Let X be a non-linear quasi-algebra. If $\lambda\in {QSp}_X  x$ for some $x\in X$, then there are infinitely many distinct elements so that quasi-spectrum of each of them contains $\lambda$.
 \end{theorem}
 \begin{proof}
 Since $X$ is a non-linear quasi-algebra, so there exist $x_1,x_2,x_3,\cdots$ such that
$$
x\lneqq x_1\lneqq x_2\lneqq x_3\lneqq\cdots
$$
by theorem \ref{chainOfsingular}. So
$$
\lambda\in {QSp}_X  x\subseteq {QSp}_X x_1\subseteq {QSp}_X x_2\subseteq {QSp}_X x_3\subseteq\cdots .
$$
 \end{proof}
\begin{problem}
In a complex Banach algebra, spectrum of any elements is nonempty \cite{2}. Therefore, we can ensure that the spectrum of regular elements is nonempty by defining a complex Banach quasi-algebra. That doesn't seem too difficult. Suppose we are able to define this space well, but can we be sure that the spectrum of singular elements is nonempty? The answer to this question is yes if every singular element has a regular element smaller than itself. But how does this happen? The following example may be useful:
\begin{example}
If $X={({\Omega}_C(\mathbb{R}))}_s\cup \lbrace\lbrace 0\rbrace\rbrace$, $a,b\in\mathbb{R}$ and $a,b\geq 0$, then ${QSp}_X[a,b]=\emptyset$, and that is because $[a,b]$ does not contain any regualar element.
\end{example}
Of course, this example was with the field of real numbers. Anyway, what condition or conditions are required for an singular element to have an regular element smaller than itself?
\end{problem}


%
%



\end{document}